\documentclass[12pt,a4paper]{article}
\usepackage[a4paper,margin=1in]{geometry}
\usepackage[T1]{fontenc}
\usepackage[utf8]{inputenc}
\usepackage{amsmath,amssymb,amsthm}
\usepackage[numbers,sort&compress]{natbib}
\usepackage{tikz}
\usetikzlibrary{calc,positioning,backgrounds,arrows.meta,decorations.markings}
\usepackage{graphicx}
\usepackage{float}
\usepackage[colorlinks=true,linkcolor=blue,citecolor=blue]{hyperref}
\usepackage{indentfirst}

\definecolor{affiliationblue}{RGB}{0,0,255}
\newcommand{\Z}{\mathbb Z}
\newcommand{\BoxProd}{\mathbin{\square}}
\newcommand{\symdiff}{\mathbin{\triangle}}
\newcommand{\modadd}[1]{\mathbin{+_{#1}}}

\newtheorem{theorem}{Theorem}[section]
\newtheorem{lemma}{Lemma}[section]

\newtheorem{corollary}[theorem]{Corollary}
\theoremstyle{definition}
\newtheorem{definition}[theorem]{Definition}
\theoremstyle{remark}
\newtheorem{remark}[theorem]{Remark}

\title{Two-disjoint-cycle-cover edge bipancyclicity in Cartesian products
	\footnote{This research is supported by the Fundamental Research Funds for the Central Universities and  Beijing Natural Science Foundation (1252010).}}
\author{ Ke Lu$^{\rm a}$  \quad Ruichao Niu$^{\rm a,}$\footnote{Corresponding author. \newline {\em E-mail address:} niuruichao@muc.edu.cn (R. Niu). }  \quad Min Xu$^{\rm b}$\\
$^{\rm a}${\small\it School of Science, Minzu University of China, Beijing, 100081, P. R. China }\\
$^{\rm b}${\small\it  School of Mathematical Sciences, Beijing Normal University,}\\
{\small\it Laboratory of Mathematics and Complex Systems,}\\
{\small\it Ministry of Education, Beijing, 100875, P. R. China}}

\date{}

\begin{document}
\maketitle

\begin{abstract}
For two even integers $r_1\leq r_2$, a graph $G$ of even order is \emph{2-DCC edge $[r_1,r_2]$-bipancyclic} if for every ordered pair $(e,f)$ of independent edges and every even integer $\ell\in[r_1,r_2]$, there exist two vertex-disjoint cycles $C_1$ and $C_2$ with $V(C_1)\cup V(C_2)=V(G)$, $|C_1|=\ell$, $|C_2|=|V(G)|-\ell$, $e\in E(C_1)$, and $f\in E(C_2)$. We study this property for Cartesian products $G=H_1\square\cdots\square H_d$ with $d\geq2$, where each factor is $K_2$ or an even-order bipartite edge-Hamiltonian graph. We prove that $G$ is 2-DCC edge $[4,|V(G)|/2]$-bipancyclic whenever $d\geq4$, and also for $d\in\{2,3\}$ under explicit Hamiltonian-cycle conditions on $H_i$. Here $1\leq i\leq d$. As applications, we give sharp characterizations for undirected toroidal meshes, star-cube networks, and hyperstar networks.
\end{abstract}

\section{Introduction}

A graph is a natural model for a network where vertices represent processors and edges represent communication links. Embedding a cycle of length $k$ into the network establishes a ring connecting $k$ processors in sequence, and the ability to embed rings of many different sizes reflects the flexibility and fault tolerance of the network. A two-disjoint-cycle cover (2-DCC) of a graph $G$ is a pair of vertex-disjoint cycles whose vertex sets partition $V(G)$, so every processor belongs to exactly one ring. When such covers exist for all admissible ring-size pairs, the network supports parallel ring-based communication.

Throughout this paper, all graphs are finite, simple, and undirected. For a graph $G$, we write $V(G)$ and $E(G)$ for its vertex set and edge set, and $|V(G)|$ for its order; $G\cong H$ means that $G$ and $H$ are isomorphic. For an edge $g=uv$, let $V(g)=\{u,v\}$ be the set of its endpoints. Two edges are \emph{independent} if they share no common endpoint, and a \emph{perfect matching} is a set of pairwise independent edges covering all vertices. For $U\subseteq V(G)$, let $G[U]$ denote the subgraph of $G$ induced by $U$. For $F\subseteq E(G)$, write $G-F$ for the graph with vertex set $V(G)$ and edge set $E(G)\setminus F$, and write $G-g=G-\{g\}$. A path on $t$ vertices is denoted by $P_t$ and a cycle on $k$ vertices by $C_k$, which is called a \emph{$k$-cycle}. Paths and cycles are written as vertex sequences, such as $C=v_0v_1\cdots v_{k-1}v_0$, and $|C|$ denotes the length of a cycle $C$. In unions and symmetric differences ($\symdiff$) of cycles, we represent each cycle by its edge set, and $\dot\cup$ denotes a disjoint union. A path or cycle is \emph{Hamiltonian} if it visits every vertex of $G$, a graph is \emph{Hamiltonian} if it contains a Hamiltonian cycle, and a graph is \emph{edge-Hamiltonian} if it is Hamiltonian and every edge lies on some Hamiltonian cycle.

\begingroup
\color{black}
\hypersetup{linkcolor=black,citecolor=black}
As a strengthening of Hamiltonicity, in 2015, Kung and Chen~\cite{KungChen2015} introduced 2-DCC pancyclicity and its edge and vertex versions. In these notions, two vertex-disjoint cycles together span the graph, while in the edge and vertex versions the two cycles are additionally required to contain prescribed independent edges or prescribed distinct vertices, respectively. For bipartite graphs, all cycles have even length, and the corresponding notions are called 2-DCC bipancyclicity, 2-DCC edge-bipancyclicity, and 2-DCC vertex-bipancyclicity. Since all graphs considered in this paper are bipartite, we give the precise definitions of these three notions below.

\begin{definition}\label{def:2dcc}\label{def:2dcc-vertex}
Let $G$ be a graph of even order $N$, and let $r_1,r_2$ be even integers with $4\leq r_1\leq r_2\leq N/2$. For an even integer $\ell\in[r_1,r_2]$, consider two vertex-disjoint cycles $C_1,C_2$ satisfying
\[
V(C_1)\mathbin{\dot\cup}V(C_2)=V(G),\qquad |C_1|=\ell,\qquad |C_2|=N-\ell.
\]
\begin{list}{}{%
\setlength{\leftmargin}{2em}%
\setlength{\labelwidth}{1.5em}%
\setlength{\labelsep}{.5em}%
\setlength{\itemsep}{2pt}%
\setlength{\parsep}{0pt}%
\setlength{\topsep}{3pt}%
\setlength{\partopsep}{0pt}%
}
\item[(i)] $G$ is \emph{2-DCC $[r_1,r_2]$-bipancyclic} if such cycles exist for every even integer $\ell\in[r_1,r_2]$.
\item[(ii)] $G$ is \emph{2-DCC edge $[r_1,r_2]$-bipancyclic} if, for every ordered pair $(e,f)$ of independent edges of $G$ and every even integer $\ell\in[r_1,r_2]$, such cycles can be chosen with $e\in E(C_1)$ and $f\in E(C_2)$.
\item[(iii)] $G$ is \emph{2-DCC vertex $[r_1,r_2]$-bipancyclic} if, for every ordered pair $(u,v)$ of distinct vertices of $G$ and every even integer $\ell\in[r_1,r_2]$, such cycles can be chosen with $u\in V(C_1)$ and $v\in V(C_2)$.
\end{list}
\end{definition}

Moreover, for a bipartite graph $G$ of minimum degree at least two, 2-DCC edge $[r_1,r_2]$-bipancyclicity implies 2-DCC vertex $[r_1,r_2]$-bipancyclicity, and the latter implies 2-DCC $[r_1,r_2]$-bipancyclicity.
\endgroup

 Since then, these properties have been studied for many network families. Kung et al.~\cite{KungChenLinHsu2021} proved all three types for locally twisted cubes. Niu et al.~\cite{NiuXuLai2021} determined 2-DCC vertex bipancyclicity for undirected toroidal meshes, including a complete description of exceptional cases. Wu and Sabir~\cite{WuSabir2023} studied spanning disjoint cycles in hypercubes with prescribed edges in each cycle. Hao et al.~\cite{HaoQinZhangChang2024} established 2-DCC pancyclicity for DCell networks. Li et al.~\cite{LiChenLu2025} proved 2-DCC pancyclicity for split-star networks. Xue et al.~\cite{XueLuQiao2025} obtained 2-DCC edge and vertex bipancyclicity for star graphs. Wei et al.~\cite{WeiHaoChang2020} proved that the $n$-dimensional balanced hypercube is 2-DCC $[4,2^{2n-1}]$-bipancyclic for $n\geq2$. Niu et al.~\cite{NiuZhouXu2023} established 2-DCC vertex $[4,2^{n-1}]$-bipancyclicity for $n$-dimensional bipartite hypercube-like networks with $n\geq3$. Zhang et al.~\cite{ZhangHaoLaiLee2023} proved 2-DCC $[4,n!/2]$-bipancyclicity for $n$-dimensional bubble-sort star graphs when $n\geq4$.

The \emph{Cartesian product} $G\square H$ of graphs $G$ and $H$ has vertex set $V(G)\times V(H)$, where $(u,v)$ and $(u',v')$ are adjacent when $u=u'$ and $vv'\in E(H)$, or when $v=v'$ and $uu'\in E(G)$~\cite{HammackImrichKlavzar2011}. This product is associative and commutative up to isomorphism. Hence, for graphs $H_1,\ldots,H_d$, the product $H_1\square\cdots\square H_d$ has vertex set $V(H_1)\times\cdots\times V(H_d)$, and two vertices $(u_1,\ldots,u_d)$ and $(v_1,\ldots,v_d)$ are adjacent if and only if they differ in exactly one coordinate $i$ and $u_iv_i\in E(H_i)$; the graphs $H_i$ are the \emph{factors} of the product. Here $1\leq i\leq d$. A product of bipartite graphs is bipartite. We write $H^{\square n}$ for the product of $n$ copies of $H$. The graphs $C_s\square P_t$ and $C_s\square C_r$ are called a \emph{cylinder} and a \emph{torus}, respectively; when $s,r\geq4$ are even, $C_s\square C_r$ is an \emph{even torus}. Day and Al-Ayyoub~\cite{DayAlAyyoub1997} studied Cartesian products as a construction for interconnection networks, and many interconnection networks used in practice are defined by Cartesian products. In this paper, we consider graphs $G=H_1\square\cdots\square H_d$ in which every factor $H_i$ is bipartite, and we prove the following result.

\begin{theorem}\label{thm:product}
Let $G=H_1\square\cdots\square H_d$, where $d\geq2$ and every $H_i$ is either $K_2$ or a bipartite even-order graph of order at least four that has a Hamiltonian cycle and in which every edge lies on a Hamiltonian cycle. Here $1\leq i\leq d$. Suppose that one of the following holds:
\begin{enumerate}
\item[(i)] $d\geq4$;
\item[(ii)] $d=3$ and at least two factors have order at least four;
\item[(iii)] $d=3$, exactly one factor has order at least four, and every set of at most two edges of that factor lies on a common Hamiltonian cycle;
\item[(iv)] $d=2$, both factors have order at least four, and every set of at most two edges of each factor lies on a common Hamiltonian cycle;
\item[(v)] $G\cong K_2\square K_2\square K_2$.
\end{enumerate}
Then $G$ is 2-DCC edge $[4,|V(G)|/2]$-bipancyclic.
\end{theorem}

Furthermore, in Section~\ref{sec:applications} we apply this theorem to give sharp characterizations for undirected toroidal meshes, star-cube networks, and hyperstar networks.

The remainder of this paper is organized as follows. Section~\ref{sec:notation} introduces projections and directions in Cartesian products and defines the three network families. Section~\ref{sec:product} proves Theorem~\ref{thm:product}. Section~\ref{sec:applications} applies the theorem to the three families and determines all exceptional cases.

\section{Preliminaries}
\label{sec:notation}

\subsection{Projections and directions}\label{subsec:projections}

Throughout this subsection, let $G=H_1\square\cdots\square H_d$ with $d\geq2$, and write a vertex of $G$ as $u=(u_1,\ldots,u_d)$ with $u_i\in V(H_i)$.

For a positive integer $n$, write $[n]=\{1,\ldots,n\}$; $\Z_s$ denotes the integers modulo $s$, and $\modadd{s}$ denotes addition modulo $s$. For $i\in[d]$, let $\pi_i:V(G)\to V(H_i)$ denote the \emph{projection} onto the $i$th coordinate, $\pi_i(u)=u_i$, and write $\pi_i(U)=\{\pi_i(v):v\in U\}$ for $U\subseteq V(G)$. Each edge $uv$ of $G$ changes exactly one coordinate $i$, called its \emph{direction}. We call \(i\) the direction of \(uv\), or equivalently say that \(uv\) is an edge in the \(H_i\)-direction. In this case $\pi_i(u)\pi_i(v)\in E(H_i)$ is the \emph{projection} of $uv$ onto $H_i$, and $\pi_j(u)=\pi_j(v)$ for every $j\neq i$.

For a nonempty $K=\{k_1<\cdots<k_t\}\subseteq[d]$, put $G_K=H_{k_1}\square\cdots\square H_{k_t}$ and define $\pi_K:V(G)\to V(G_K)$ by $\pi_K(u_1,\ldots,u_d)=(u_{k_1},\ldots,u_{k_t})$. For disjoint nonempty $I,J\subseteq[d]$ with $I\cup J=[d]$, the map $u\mapsto(\pi_I(u),\pi_J(u))$ is a graph isomorphism from $G$ to $G_I\square G_J$. Under this identification, if $D_I$ and $D_J$ are Hamiltonian cycles in $G_I$ and $G_J$, respectively, then $D_I\square D_J$ is a spanning subgraph of $G$.

Given the two prescribed edges $e,f$ of $G$ and a nonempty $K\subseteq[d]$, let
\[
F_K=\{\pi_K(u)\pi_K(v):uv\in\{e,f\},\ \text{direction of }uv\in K\}\subseteq E(G_K)
\]
be the set of projections onto $G_K$ of those prescribed edges whose direction lies in $K$. In particular, $|F_K|\leq2$.

\subsection{Three families of networks}\label{subsec:families}

We now define the three families of interconnection networks considered in Section~\ref{sec:applications}. See Figure~\ref{fig:examples} for examples.

\begin{figure}[ht]
\centering
\resizebox{\linewidth}{!}{\begin{tikzpicture}[
    vertex/.style={circle, fill=black, inner sep=1.3pt, outer sep=0pt},
    baseedge/.style={draw=black!85, line width=0.65pt},
    wrapedge/.style={draw=blue!70!black, line width=0.55pt, densely dashed},
    crossedge/.style={draw=black!45, line width=0.45pt},
    groupbox/.style={draw=black!25, rounded corners=4pt, dashed, fill=none},
    lbl/.style={font=\footnotesize},
    panelcaption/.style={font=\small}
]

\begin{scope}[shift={(0,0)}, scale=0.75]
    \foreach \k/\dx/\dy/\op in {1/0.45/0.38/0.45, 0/0/0/1.0} {
        \begin{scope}[shift={(\dx*1cm, \dy*1cm)}, opacity=\op]
            \foreach \i in {0,1,2,3} {
                \draw[wrapedge] (\i,0) to[bend right=24] (\i,3);
                \draw[wrapedge] (0,\i) to[bend left=24] (3,\i);
            }
            \foreach \i in {0,1,2} {
                \foreach \j in {0,1,2,3} {
                    \draw[baseedge] (\i,\j) -- (\i+1,\j);
                    \draw[baseedge] (\j,\i) -- (\j,\i+1);
                }
            }
            \foreach \i in {0,1,2,3} {
                \foreach \j in {0,1,2,3} {
                    \node[vertex] (v\i\j\k) at (\i,\j) {};
                }
            }
        \end{scope}
    }
    \foreach \i in {0,1,2,3} {
        \foreach \j in {0,1,2,3} {
            \draw[crossedge] (v\i\j0) -- (v\i\j1);
        }
    }
    \node[lbl, below left=-1pt] at (v000) {$000$};
    \node[lbl, below right=-1pt] at (v300) {$300$};
    \node[lbl, above left=-1pt] at (v030) {$030$};
    \node[lbl, above right=-1pt] at (v331) {$331$};
    \node[panelcaption] at (1.7, -1.3) {(a) $C(4,4,2)$};
\end{scope}

\begin{scope}[shift={(5.5,0)}, scale=0.75]
    \foreach \cid/\cx/\cy/\bstr in {00/0/0/00, 10/3.2/0/10, 01/0/3.2/01, 11/3.2/3.2/11} {
        \draw[groupbox] (\cx,\cy) circle (1.15cm);
        \node[lbl, text=black!60] at (\cx,\cy) {$S_3 \times \{\bstr\}$};
        \foreach \a in {0,1,2,3,4,5} {
            \node[vertex] (sc_\cid_\a) at ($(\cx,\cy) + (\a*60+30:0.88cm)$) {};
        }
        \foreach \a [evaluate=\a as \nexta using {int(mod(\a+1,6))}] in {0,1,2,3,4,5} {
            \draw[baseedge] (sc_\cid_\a) -- (sc_\cid_\nexta);
        }
    }
    \begin{scope}[on background layer]
        \foreach \a in {0,1,2,3,4,5} {
            \draw[crossedge] (sc_00_\a) -- (sc_10_\a);
            \draw[crossedge] (sc_01_\a) -- (sc_11_\a);
            \draw[crossedge] (sc_00_\a) -- (sc_01_\a);
            \draw[crossedge] (sc_10_\a) -- (sc_11_\a);
        }
    \end{scope}
    \node[panelcaption] at (1.6, -1.7) {(b) $SC_{3,2} = S_3 \Box Q_2$};
\end{scope}

\begin{scope}[shift={(12.8, 1.2)}, scale=0.75]
    \foreach \c in {0,1,2,3,4,5} {
        \pgfmathsetmacro{\cx}{2.15*cos(\c*60+30)}
        \pgfmathsetmacro{\cy}{2.15*sin(\c*60+30)}
        \draw[groupbox] (\cx,\cy) circle (0.72cm);
        \foreach \v in {0,1,2,3,4,5} {
            \node[vertex] (hs_\c_\v) at ($(\cx,\cy) + (\v*60+30:0.52cm)$) {};
        }
        \foreach \v [evaluate=\v as \nextv using {int(mod(\v+1,6))}] in {0,1,2,3,4,5} {
            \draw[baseedge] (hs_\c_\v) -- (hs_\c_\nextv);
        }
    }
    \begin{scope}[on background layer]
        \foreach \c [evaluate=\c as \nextc using {int(mod(\c+1,6))}] in {0,1,2,3,4,5} {
            \foreach \v in {0,1,2,3,4,5} {
                \draw[crossedge] (hs_\c_\v) -- (hs_\nextc_\v);
            }
        }
    \end{scope}
    \node[panelcaption] at (0, -3.3) {(c) $S_{3,3} = S_3 \Box S_3$};
\end{scope}

\end{tikzpicture}
}%
\caption{Toroidal mesh (a), star-cube (b), and hyperstar (c) networks.}
\label{fig:examples}
\label{fig:mesh}
\label{fig:s4}
\label{fig:hyperstar}
\end{figure}

\begin{definition}\label{def:mesh}
An \emph{undirected toroidal mesh} is a Cartesian product of at least two factors, each either $K_2$ or a cycle.
For integers $d_1,\ldots,d_n\geq2$, put
\[
C(d_1,\ldots,d_n)=F_1\square\cdots\square F_n,
\qquad
F_i=\begin{cases}K_2,&d_i=2,\\ C_{d_i},&d_i\geq3.\end{cases}
\]
Here $n\geq2$ and $1\leq i\leq n$. 
If every $d_i=2$, then $C(d_1,\ldots,d_n)=Q_n$. If every $d_i=k\geq3$, it is the $k$-ary $n$-cube. In this paper, we consider only bipartite undirected toroidal meshes, which means that all $d_i$ are even.
\end{definition}

\begin{definition}\label{def:starcube}
The permutation star graph $S_m$ is the Cayley graph of the symmetric group $\mathfrak S_m$ generated by the transpositions $(1\ i)$, $2\leq i\leq m$~\cite{AkersKrishnamurthy1989}. For $m\geq3$ and $n\geq1$, the \emph{star-cube network} is $\mathit{SC}_{m,n}=S_m\square Q_n$, where $Q_n=K_2^{\square n}$.
\end{definition}

\begin{lemma}[\cite{XueLuQiao2025}, Corollary~2.2]\label{lem:star-two-edges}
Let $m\geq4$. If $e_1$ and $e_2$ are distinct edges of $S_m$, then $S_m$ has a Hamiltonian cycle containing $e_1$ and $e_2$.
\end{lemma}

\begin{definition}\label{def:hyperstar}
For $r\geq1$ and $m_1,\ldots,m_r\geq3$, the \emph{hyperstar network} is
\[
S_{m_1,\ldots,m_r}=S_{m_1}\square\cdots\square S_{m_r},
\qquad N=\prod_{i=1}^r m_i!.
\]
Here $N=|V(S_{m_1,\ldots,m_r})|$.
\end{definition}

\section{The proof of main result}\label{sec:product}

The goal of this section is to establish 2-DCC edge bipancyclicity for Cartesian products whose factors satisfy one of five Hamiltonian-cycle conditions. The proof reduces to two steps: finding a spanning even torus $C_s\square C_r$ inside the product that contains both prescribed edges, and then using the torus lemma below to produce two vertex-disjoint cycles of the required lengths inside that torus.

\subsection{Cancatenation}\label{cancatenation}
All the proofs are based on the following simple operation named \textit{cancatenation} in a two-dimensional torus.

Write $T=C_s\square C_r$, where $s,r\geq4$ are even. Its vertices are $x_j^i$ with column $j\in\mathbb Z_s$ and row $i\in\mathbb Z_r$. The row $i$ is the cycle $C_s^i=x_0^ix_1^i\cdots x_{s-1}^ix_0^i$. Subscripts and superscripts are read modulo $s$ and $r$, respectively. In these coordinates, $\pi_2(x_j^i)=i$.

We use the following two concatenations.

We first define a 4-cycle for \(C_s\BoxProd P_t\) and  \(C_s\BoxProd C_r\).
\begin{equation}\label{eq:row-square}
 R_{i,h}(c)=x_c^i x_{c\modadd{s}1}^i
            x_{c\modadd{s}1}^h x_c^h x_c^i.
\end{equation}
Here, $i$ and $h$ are adjacent rows and $c$ is a column. The \emph{starting column} of a horizontal edge $x_c^i x_{c\modadd{s}1}^i$ is $c$.
For example, \(R_{0,1}(0)\) refers to the 4-cycle
\(x_0^0x_1^0x_1^1x_0^1x_0^0\).

We write \(a=u_iu_{i+1}, b=v_iv_{i+1}, F=\{u_iv_i,u_{i+1}v_{i+1}\}\), where \(u_iu_{i+1}v_{i+1}v_iu_i\) is a 4-cycle \(R\); here \(a\) and \(b\) are opposite edges of \(R\) and the two edges of \(F\) are the remaining two edges of \(R\). The two concatenations are as follows.

(I) If \(C_1,C_2\) are vertex-disjoint cycles with \(a\in E(C_1)\) and \(b\in E(C_2)\), then
\[
E(C'):=\bigl(E(C_1)\cup E(C_2)\bigr) \mathbin{\triangle} E(R)=\bigl((E(C_1)\cup E(C_2))\setminus\{a,b\}\bigr)\cup F
\]
defines a cycle on \(V(C_1)\cup V(C_2)\) of length \(|C_1|+|C_2|\). If \(P\subseteq E(C_1)\cup E(C_2)\) is a set of prescribed edges and \(P\cap\{a,b\}=\varnothing\), then \(P\subseteq E(C')\).

(II) If \(a\in E(C)\) and \(V(b)\cap V(C)=\varnothing\), then
\[
E(C'):=E(C)\mathbin{\triangle} E(R)=(E(C)\setminus\{a\})\cup\{b\}\cup F
\]
defines a cycle on \(V(C)\cup V(b)\) of length \(|C|+2\). If \(P\subseteq E(C)\) is a set of prescribed edges and \(a\notin P\), then \(P\subseteq E(C')\).

The first operation merges two cycles into one cycle of length \(|C_1|+|C_2|\). The second increases the length of a cycle by two.

Figure~\ref{fig:concat} shows the two concatenations. 

\begin{figure}[ht]
\centering
\input{figures/extracted/concat.tex}%
\caption{Concatenation (I) and Concatenation (II).}
\label{fig:concat}
\end{figure}

\subsection{2DCC even torus lemma}

In \(C_s\BoxProd P_t\) with even \(s\geq4\) and \(c\in\Z_s\), we define the perfect matching
\begin{equation}\label{eq:matching-case2}
 \mathcal M_i(c)=
 \{x_{c\modadd{s}(2j)}^i x_{c\modadd{s}(2j+1)}^i:j=0,\ldots,s/2-1\}.
\end{equation} Here $0\leq i<t$.
\begin{lemma}\label{lem:cylinder}
Let \(s\geq4\) be even and \(t\geq1\).
If \(M\) is a perfect matching of \(C_s^0\), and
\(g\in E(C_s\BoxProd P_t)\), then \(C_s\BoxProd P_t\) has a
Hamiltonian cycle \(H\) with \(M\cup\{g\}\subseteq E(H)\).
\end{lemma}

\begin{proof}
Without loss of generality, we may assume that \(M=\mathcal M_0(0)\). Otherwise, we can shift the columns.

For \(t=1\), take \(H=C_s^0\), which contains both the matching and \(g\). Assume that \(t\geq2\). To ensure prescribed edges are included, we choose an odd \(c\in\Z_s\) such that
\[
\begin{cases}
 c\in\{b\modadd{s}(-i),b\modadd{s}(-i-1)\},
     &g=x_b^ix_b^{i+1}, b\in\Z_s \quad 0\leq i<t-1\\
 c\notin\{b\modadd{s}(-i),b\modadd{s}(1-i)\},
     &g=x_b^ix_{b\modadd{s}1}^i, b\in\Z_s \quad 0\leq i<t
\end{cases}
\]
In the former case, these residues are consecutive modulo the even integer \(s\), so
exactly one is odd. In the latter case, \(\Z_s\) has
\(s/2\geq2\) odd elements. Hence such a \(c\) exists.

Construct
\begin{equation}\label{eq:cylinder-cycle}
 H_c=\left(\bigcup_{i=0}^{t-1}C_s^i\right)
      \symdiff\left(\bigcup_{i=0}^{t-2}
         R_{i,i+1}(c\modadd{s}i)\right).
\end{equation}
Apply the concatenations in~\eqref{eq:cylinder-cycle} one by one, from row \(0\) to row \(t-1\).
At each step, the next row cycle is disjoint from the cycle already obtained.
For \(i\ge 1\), the concatenation at row \(i\) uses the edge starting at column
\(c+_s i\), while the previous 4-cycle removed the edge starting at column
\(c+_s(i-1)\). These two edges are different. Hence concatenation (I) can be
applied at each step, and we obtain a cycle of length \(st\).

Only the edge starting at column \(c\) is removed from row \(0\).
Since \(c\) is odd, this edge does not belong to \(\mathcal M_0(0)\).
If \(g=x_b^i x_b^{i+1}\), then our choice of \(c\) gives
\(b=c+_s i\) or \(b=c+_s(i+1)\), so \(g\) is added in the construction.
If \(g=x_b^i x_{b+_s1}^i\), then it can be removed only when
\(b=c+_s i\) or \(b=c+_s(i-1)\). Our choice of \(c\) rules out both equalities.
Thus \(\mathcal M_0(0)\cup\{g\}\subseteq E(H_c)\).
\end{proof}

\begin{figure}[ht]
\centering
\input{figures/extracted/cylinder.tex}%
\caption{The cylinder construction preserves matching $M$ and edge $g$.}
\label{fig:cylinder}
\end{figure}

Note that for $C_s\square P_t$, this lemma can be applied to both $C^0_s$ and $C^{t-1}_s$. Throughout this paper, we will omit the explaination of relabeling row index when the $C_s\square P_t$ is clear.

Next, we apply Lemma~\ref{lem:cylinder} to construct two cycles in an even torus $T=C_s\square C_r$. An edge of $T$ is \emph{horizontal} if its endpoints have the same row index, and \emph{vertical} otherwise. Before considering an abitrary pair of edges, we introduce a simple observation that can substaintially reduce cases. We call it as the  canonical form.

\textbf{Claim.} Let $s,r\geq4$ be even and let $e,f$ be independent edges of $T=C_s\square C_r$. After exchanging the two factors and relabelling $s,r$ if necessary, there is an automorphism of $T$ mapping $e$ to $x_0^0x_1^0$ (horizontal) or $x_0^0x_0^1$ (vertical), and mapping $f$ to an edge satisfying
\begin{equation}\label{eq:torus-normalization}
 V(f)\subseteq\bigcup_{i=r/2}^{r-1}V(C_s^i).
\end{equation}

\textbf{Proof of claim.}
We split this normalization into three parts. 

(i) By changing $C_s\square C_r$ to $C_r\square C_s$ when necessary, we can ensure that $e$ and $f$ satisfy that \(\pi_2(V(e))\cap\pi_2(V(f))=\varnothing\) since $e$ and $f$ are independent. Indeed, in two-dimensional space, if two edges have the same projections in both horizontal and vertical direction, they must have at least one endpoint. Figure~\ref{fig:coordinate-separation} proves that intuitively.

(ii) Choose a proper automorphism, we can shift $e$ to $x_0^0x_1^0$ or $x_0^0x_0^1$ by symmetry. 

(iii) If \eqref{eq:torus-normalization} holds, we are done. Otherwise, apply one of the following reflections.

If $e=x_0^0x_1^0$, then $0\notin\pi_2(V(f))$, and \eqref{eq:torus-normalization} failing implies $\pi_2(V(f))\subseteq\{1,\ldots,r/2\}$. The map $x_j^i\mapsto x_j^{(-i)\bmod r}$ fixes $e$ and sends $\pi_2(V(f))$ into $\{r/2,\ldots,r-1\}$.

If $e=x_0^0x_0^1$, then $\pi_2(V(f))\cap\{0,1\}=\varnothing$, and \eqref{eq:torus-normalization} failing implies $\pi_2(V(f))\subseteq\{2,\ldots,r/2\}$. The map $x_j^i\mapsto x_j^{(1-i)\bmod r}$ fixes $e$ and sends $\pi_2(V(f))$ into $\{r/2+1,\ldots,r-1\}$.

In both cases \eqref{eq:torus-normalization} holds after the reflection.

\begin{figure}[!ht]
\centering
\resizebox{0.98\linewidth}{!}{%
\begingroup%
\definecolor{figcycleone}{rgb}{0,0,1}%
\definecolor{figcycletwo}{rgb}{1,0,1}%
\definecolor{figmatching}{rgb}{0.20,0.57,0.10}%
\definecolor{figdeleted}{rgb}{1,0,0}%
\definecolor{figguide}{gray}{0.68}%
\setlength{\fboxsep}{8pt}%
\setlength{\fboxrule}{0pt}%
\raisebox{\depth}{\fbox{%
\begin{tikzpicture}[
  x=1.25cm,y=1.05cm,
  font=\small,
  grid edge/.style={draw=figguide!45,line width=.45pt},
  projection e/.style={draw=figcycleone!65,densely dotted,line width=.7pt},
  projection f/.style={draw=figcycletwo!65,densely dotted,line width=.7pt},
  edge e/.style={draw=figcycleone,line width=2.2pt,line cap=round},
  edge f/.style={draw=figcycletwo,line width=2.2pt,line cap=round},
  axis/.style={draw=black!50,line width=.55pt},
  coordinate label/.style={font=\footnotesize,text=black!65}
]

\begin{scope}
  \node at (1.2,4.03) {\textbf{(a)}\quad Row separation};
  \foreach \j in {0,...,3} {
    \draw[grid edge] (\j,0)--(\j,3);
    \draw[grid edge] (0,\j)--(3,\j);
    \foreach \i in {0,...,3} {
      \fill[black!40] (\j,\i) circle[radius=1.15pt];
    }
    \node[coordinate label] at (\j,-.32) {$\j$};
  }

  \draw[axis] (-.92,-.12)--(-.92,3.2);
  \node at (-.92,3.58) {$\pi_2$};
  \foreach \i in {0,...,3} {
    \draw[axis] (-.98,\i)--(-.86,\i);
    \node[coordinate label,anchor=east] at (-1.09,\i) {$\i$};
  }

  \draw[projection e] (-.92,1)--(1,1);
  \draw[projection f] (-.92,2)--(2,2);
  \draw[projection f] (-.92,3)--(2,3);

  \draw[edge e] (1,1)--(2,1);
  \draw[edge f] (2,2)--(2,3);
  \foreach \j/\i in {1/1,2/1} {
    \fill[figcycleone] (\j,\i) circle[radius=2.6pt];
  }
  \foreach \i in {2,3} {
    \draw[draw=figcycletwo,fill=white,line width=1.2pt]
      (2,\i) ++(-2.1pt,-2.1pt) rectangle ++(4.2pt,4.2pt);
  }
  \fill[figcycleone] (-.92,1) circle[radius=2.8pt];
  \foreach \i in {2,3} {
    \draw[draw=figcycletwo,fill=white,line width=1.2pt]
      (-.92,\i) ++(-2.2pt,-2.2pt) rectangle ++(4.4pt,4.4pt);
  }
  \node[text=figcycleone,above] at (1.5,1.08) {$e$};
  \node[text=figcycletwo,right] at (2.1,2.5) {$f$};
  \node at (1.2,-1.65)
    {$\pi_2(V(e))\cap\pi_2(V(f))=\varnothing$};
\end{scope}

\begin{scope}[shift={(6.95,0)}]
  \node at (1.5,4.03) {\textbf{(b)}\quad Column separation};
  \foreach \j in {0,...,3} {
    \draw[grid edge] (\j,0)--(\j,3);
    \draw[grid edge] (0,\j)--(3,\j);
    \foreach \i in {0,...,3} {
      \fill[black!40] (\j,\i) circle[radius=1.15pt];
    }
    \node[coordinate label,anchor=east] at (-.23,\j) {$\j$};
  }

  \draw[axis] (-.12,-.8)--(3.2,-.8);
  \node[anchor=west] at (3.34,-.8) {$\pi_1$};
  \foreach \j in {0,...,3} {
    \draw[axis] (\j,-.86)--(\j,-.74);
    \node[coordinate label] at (\j,-1.12) {$\j$};
  }

  \draw[projection e] (1,-.8)--(1,1);
  \draw[projection e] (2,-.8)--(2,1);
  \draw[projection f] (3,-.8)--(3,1);

  \draw[edge e] (1,1)--(2,1);
  \draw[edge f] (3,1)--(3,2);
  \foreach \j/\i in {1/1,2/1} {
    \fill[figcycleone] (\j,\i) circle[radius=2.6pt];
  }
  \foreach \i in {1,2} {
    \draw[draw=figcycletwo,fill=white,line width=1.2pt]
      (3,\i) ++(-2.1pt,-2.1pt) rectangle ++(4.2pt,4.2pt);
  }
  \foreach \j in {1,2} {
    \fill[figcycleone] (\j,-.8) circle[radius=2.8pt];
  }
  \draw[draw=figcycletwo,fill=white,line width=1.2pt]
    (3,-.8) ++(-2.2pt,-2.2pt) rectangle ++(4.4pt,4.4pt);
  \node[text=figcycleone,above] at (1.5,1.08) {$e$};
  \node[text=figcycletwo,right] at (3.1,1.5) {$f$};
  \node at (1.5,-1.65)
    {$\pi_1(V(e))\cap\pi_1(V(f))=\varnothing$};
\end{scope}
\end{tikzpicture}%
}}%
\endgroup%
}%
\caption{Independent edges have disjoint row (a) or column (b) projections.}
\label{fig:coordinate-separation}
\end{figure}

Without loss of generality, we may assume that even torus and prescribed edges are canonical.

\begin{lemma}\label{lem:even-torus}\label{lem:torus}
Let \(s,r\geq4\) be even. Then \(T=C_s\BoxProd C_r\) is
2-DCC edge \([4,sr/2]\)-bipancyclic.
\end{lemma}

\begin{proof}

Let \(e,f\) be independent edges of \(T\) in the canonical forms above, and let \(\ell\) be even with \(4\leq\ell\leq sr/2\).

\smallskip
\noindent\textbf{Case 1. \(4\leq\ell\leq s\)}

Let \(k=\ell/2\) (so \(2\leq k\leq s/2\) and \(s-k\geq2\)), and define
\begin{equation}\label{eq:torus-case1-cycles}
\begin{aligned}
 J_1&=x_0^0x_1^0\cdots x_{k-1}^0
       x_{k-1}^1x_{k-2}^1\cdots x_0^1x_0^0,\\
 K&=x_k^0x_{k+1}^0\cdots x_{s-1}^0
       x_{s-1}^1x_{s-2}^1\cdots x_k^1x_k^0.
\end{aligned}
\end{equation}

Then \(e\in E(J_1)\) and \(|J_1|=2k=\ell\). 

\(J_1\) and \(K\) are disjoint cycles, and \(x_k^bx_{k+1}^b\in E(K)\) for either \(b=0\) or \(b=1\).
Put
\begin{equation}\label{eq:torus-case1-D}
 D=T\!\left[\bigcup_{i=2}^{r-1}V(C_s^i)\right]
   \cong C_s\BoxProd P_{r-2},
\end{equation}
and choose
\begin{equation}\label{eq:torus-case1-choice}
 (h,b)=\begin{cases}
 (2,1),&f\ne x_k^2x_{k+1}^2,\\
 (r-1,0),&f=x_k^2x_{k+1}^2.
 \end{cases}
\end{equation}
Since \(r/2\geq2\), \eqref{eq:torus-normalization} puts \(f\) in \(E(D)\). Apply Lemma~\ref{lem:cylinder} with \(\mathcal M_h(k)\) and \(f\) to give a Hamiltonian cycle \(H\) of \(D\). The resulting \(H\) contains \(f\) and
\(x_k^hx_{k+1}^h\in\mathcal M_h(k)\). The rows \(b,h\) are adjacent in both cases. Therefore, we can define
\[
 J_2=(K\cup H)\symdiff R_{b,h}(k).
\]
The edge \(x_k^hx_{k+1}^h\) differs from \(f\) by the choice of \((h,b)\) when \(h=2\), and when \(h=r-1\geq3\), \(f\) lies in row \(2\).
The edge removed from \(K\) also differs from \(f\), since \(f\in E(D)\). Concatenation (I) therefore yields a cycle \(J_2\) containing \(f\),
with \(V(J_2)=V(T)\setminus V(J_1)\) and \(|J_2|=2(s-k)+(r-2)s=sr-\ell\).

\begin{figure}[ht]
\centering
\input{figures/extracted/torus-case1.tex}%
\caption{Concatenation constructs $J_2$ through $f$, disjoint from $J_1$ through $e$.}
\label{fig:torus-case1}
\end{figure}

\medskip
\begingroup
\noindent\textbf{Case 2. \(s<\ell\leq sr/2\)}

Set
\[
 a=\lceil\ell/s\rceil-1,\qquad q=(\ell-as)/2.
\]
So \(1\leq a\leq r/2-1, 1\leq q\leq s/2, \ell=as+2q.\)

\medskip
\noindent\textbf{Subcase 2.1. \(a\geq2\)}

Define
\begin{equation}\label{eq:torus-case2-blocks}
 L^-=T\!\left[\bigcup_{i=0}^{a-1}V(C_s^i)\right],\qquad
 L^+=T\!\left[\bigcup_{i=a+1}^{r-1}V(C_s^i)\right].
\end{equation}
Therefore \(L^-\cong C_s\BoxProd P_a\) and \(L^+\cong C_s\BoxProd P_{r-a-1}\), with row \(a\) between them. Since \(a\geq2\), then \(e\in E(L^-)\) in both canonical forms. Since \(a+1\leq r/2\), \eqref{eq:torus-normalization} gives \(f\in E(L^+)\).

For $d={0,1...s-1}$, choose
\begin{equation}\label{eq:column-choice}
 c=\begin{cases}
 d\modadd{s}1,&f=x_d^{a+1}x_{d\modadd{s}1}^{a+1},\\
 0,&f\neq x_d^{a+1}x_{d\modadd{s}1}^{a+1}.
 \end{cases}
\end{equation}
Apply Lemma~\ref{lem:cylinder} to \(L^-\) to give a Hamiltonian cycle \(H^-\) containing \(\mathcal M_{a-1}(c)\cup\{e\}\). Apply Lemma~\ref{lem:cylinder} to \(L^+\) to obtain a Hamiltonian cycle \(H^+\) containing \(\mathcal M_{a+1}(c)\cup\{f\}\).

For \(j=0,\ldots,s/2-1\), define
\[
 R_j^-=R_{a-1,a}\bigl(c\modadd{s}(2j)\bigr),\qquad
 R_j^+=R_{a+1,a}\bigl(c\modadd{s}(2j)\bigr),
\]
and set
\begin{equation}\label{eq:J1J2-case2}
 J_1=H^-\symdiff\bigcup_{j=0}^{q-1}R_j^-,\qquad
 J_2=H^+\symdiff\bigcup_{j=q}^{s/2-1}R_j^+.
\end{equation}
We call the sets
\[
\bigl\{c\modadd{s}(2j),\,c\modadd{s}(2j+1)\bigr\},
\qquad j=0,\ldots,s/2-1,
\]
the \emph{column pairs}. These pairs partition \(\mathbb Z_s\). For \(0\le j<q\), concatenation (II) with \(R_j^-\) adds
\(x^a_{c+_s(2j)}\) and \(x^a_{c+_s(2j+1)}\) to \(H^-\).
For \(q\le j<s/2\), concatenation (II) with \(R_j^+\) adds
\(x^a_{c+_s(2j)}\) and \(x^a_{c+_s(2j+1)}\) to \(H^+\).
Therefore, the resulting cycles \(J_1\) and \(J_2\) are disjoint and together cover all vertices of row \(a\).
If \(q=s/2\), take \(J_2=H^+\).

It follows that \(J_1,J_2\) are cycles with
\[
\begin{aligned}
 V(J_1)&=V(L^-)\cup\{x_{c\modadd{s}t}^a:t\in\Z,\ 0\leq t<2q\},\\
 V(J_2)&=V(L^+)\cup\{x_{c\modadd{s}t}^a:t\in\Z,\ 2q\leq t<s\}.
\end{aligned}
\]
Clearly, two cycles \(J_1\) and \(J_2\) satisfy \(V(J_1)\mathbin{\dot\cup}V(J_2)=V(T), |J_1|=as+2q=\ell, |J_2|=(r-a-1)s+(s-2q)=sr-\ell.\)

\begin{figure}[ht]
\centering
\resizebox{0.98\linewidth}{!}{%
\begin{tikzpicture}[
  x=1mm,y=1mm,font=\normalsize,
  vertex/.style={circle,fill=black,draw=black,inner sep=0pt,minimum size=1.6mm},
  point label/.style={font=\normalsize,inner sep=0pt,outer sep=0pt,text height=3.5mm,text depth=2mm,anchor=west},
  first cycle/.style={draw=blue,line width=1pt,line cap=round,line join=round},
  second cycle/.style={draw=magenta,line width=1pt,line cap=round,line join=round},
  deleted edge/.style={draw=black,line width=.65pt,dash pattern=on 3pt off 2pt},
  region/.style={draw=black,line width=.35pt,dash pattern=on 3.5pt off 2.5pt}
]
\foreach \row/\yy in {0/48,1/40,2/32,3/24,4/12,5/4} {
  \coordinate (A\row) at (20,\yy);
  \coordinate (B\row) at (59,\yy);
  \coordinate (C\row) at (98,\yy);
  \coordinate (D\row) at (137,\yy);
}

\draw[region] (0,-3) rectangle (31,54);
\draw[region] (47,-3) rectangle (110,54);
\draw[region] (125,-3) rectangle (157,54);
\node at (15.5,58) {$L^-$};
\node at (78.5,58) {$C_s^a$};
\node at (141,58) {$L^+$};

\foreach \upper/\lower in {0/1,2/3,4/5} {
  \draw[first cycle] (A\upper)--(B\upper)--(B\lower)--(A\lower);
  \draw[second cycle] (D\upper)--(C\upper)--(C\lower)--(D\lower);
  \draw[deleted edge] (A\upper)--(A\lower);
  \draw[deleted edge] (D\upper)--(D\lower);
}
\draw[first cycle] (A1)--(A2);
\draw[second cycle] (D1)--(D2);
\draw[first cycle] (A3)--(20,21);
\draw[first cycle] (20,15)--(A4);
\draw[second cycle] (D3)--(137,21);
\draw[second cycle] (137,15)--(D4);
\foreach \xx in {20,137} {
  \foreach \yy in {16.7,18,19.3} {
    \fill[black] (\xx,\yy) circle[radius=.18mm];
  }
  \foreach \yy in {44,28,8} {
    \draw[red,line width=.75pt,line cap=round]
      (\xx-.75,\yy-.75)--(\xx+.75,\yy+.75)
      (\xx-.75,\yy+.75)--(\xx+.75,\yy-.75);
  }
}

\draw[first cycle,rounded corners=2mm]
  (A0)--(20,52)--(2,52)--(2,0)--(20,0)--(A5);
\draw[second cycle,rounded corners=2mm]
  (D0)--(137,52)--(155,52)--(155,0)--(137,0)--(D5);
\node[text=blue] at (8.5,18) {$H^-$};
\node[text=magenta] at (148.5,18) {$H^+$};
\node[text=blue] at (39.5,51.5) {$J_1$};
\node[text=magenta] at (117.5,51.5) {$J_2$};

\foreach \yy/\jj in {44/0,28/1,8/{q-1}} {
  \node at (39.5,\yy) {$R^-_{\jj}$};
}
\foreach \yy/\jj in {44/{s/2-1},28/{s/2-2},8/q} {
  \node at (117.5,\yy) {$R^+_{\jj}$};
}

\foreach \column in {A,B,C,D} {
  \foreach \row in {0,...,5} {
    \node[vertex] at (\column\row) {};
  }
}

\foreach \yy/\cc in {48/c,40/{c+1},32/{c+2},24/{c+3},12/{c+2q-2},4/{c+2q-1}} {
  \node[point label] at (4.5,\yy) {$x^{a-1}_{\cc}$};
  \node[point label] at (61,\yy) {$x^a_{\cc}$};
}
\foreach \yy/\cc in {48/{c-1},40/{c-2},32/{c-3},24/{c-4},12/{c+2q+1},4/{c+2q}} {
  \node[point label] at (82.5,\yy) {$x^a_{\cc}$};
  \node[point label] at (139,\yy) {$x^{a+1}_{\cc}$};
}
\end{tikzpicture}%
}%
\caption{Square concatenations extend $H^-$ and $H^+$ into disjoint cycles $J_1$ and $J_2$.}
\label{fig:torus-case2}
\end{figure}

Finally, the edges deleted from \(H^-\) are horizontal edges in row \(a-1\geq1\). None is \(e\) because \(e\) is either horizontal in row \(0\) or vertical, so \(e\in E(J_1)\). The edges deleted from \(H^+\) belong to \(\mathcal M_{a+1}(c)\). If \(f\notin E(C_s^{a+1})\), no such edge is \(f\). Otherwise, \(c=d\modadd{s}1\) has parity opposite to \(d\) since \(s\) is even. The matching \(\mathcal M_{a+1}(c)\) contains exactly the horizontal edges whose starting columns have the same parity as \(c\), so it does not contain \(f\). Thus \(f\in E(J_2)\), and \(J_1,J_2\) give the required 2-DCC.

\medskip
\noindent\textbf{Subcase 2.2. \(a=1\)} and \textbf{\(e=x_0^0x_1^0\)}

Here \(\ell=s+2q\). Let
\[
 L^-=C_s^0,\qquad
 L^+=T\!\left[\bigcup_{i=2}^{r-1}V(C_s^i)\right]
       \cong C_s\BoxProd P_{r-2}.
\]
Row \(1\) lies between them. Since \(r/2\geq2\), \eqref{eq:torus-normalization} gives \(f\in E(L^+)\).

If \(f\in E(C_s^2)\), write \(f=x_d^2x_{d\modadd{s}1}^2\) with \(d\in\{0,\ldots,s-1\}\). Choose
\[
 c=\begin{cases}
 d\modadd{s}1,&f\in E(C_s^2)\text{ and }d\text{ is even},\\
 d,&f\in E(C_s^2)\text{ and }d\text{ is odd},\\

 \end{cases}
\]
If \(f\notin E(C_s^2)\), then we choose \(c=1\). Thus every choice of $c$ is odd. Apply Lemma~\ref{lem:cylinder} to $L^+$ to obtain a Hamiltonian cycle $H^+$ containing $\mathcal M_2(c)\cup\{f\}$. Take $H^-=C_s^0=L^-$, which contains $\mathcal M_0(c)$.

For \(j=0,\ldots,s/2-1\), define
\[
 R_j^-=R_{0,1}\bigl(c\modadd{s}(2j)\bigr),\qquad
 R_j^+=R_{2,1}\bigl(c\modadd{s}(2j)\bigr),
\]
and set
\[
 J_1=H^-\symdiff\bigcup_{j=0}^{q-1}R_j^-,\qquad
 J_2=H^+\symdiff\bigcup_{j=q}^{s/2-1}R_j^+.
\]

It is similar to Subcase 2.1: the same concatenation-(II) argument shows that $J_1$ and $J_2$ are cycles and that the column-pair partition of $\mathbb Z_s$ assigns every vertex of row~$1$ to exactly one of them; the length calculation is identical. To check the prescribed edges, first observe that \(e=x_0^0x_1^0\) does not belong to \(\mathcal M_0(c)\), since its starting column \(0\) is even but \(c\) is odd. All edges deleted from \(H^-\) belong to this matching, so \(e\in E(J_1)\). The only edges deleted from \(H^+\) are
\[
 \{x_{c\modadd{s}(2j)}^2x_{c\modadd{s}(2j+1)}^2:q\leq j<s/2\}.
\]
If \(f\notin E(C_s^2)\), it is outside this set. If \(f\in E(C_s^2)\) and \(d\) is even, then \(c=d\modadd{s}1\) is odd, so \(f\notin\mathcal M_2(c)\) and is not deleted. If \(d\) is odd, then \(c=d\), making \(f\) the matching edge with index \(j=0\). Since \(q\geq1\), this index is outside the deletion range \(q\leq j<s/2\). Thus \(f\in E(J_2)\) in every case, and \(J_1,J_2\) give the required 2-DCC.

\medskip
\noindent\textbf{Subcase 2.3. \(a=1\)} and \textbf{\(e=x_0^0x_0^1\)}

Here \(\ell=s+2q\). Let
\[
 L^-=C_s^0,\qquad
 L^+=T\!\left[\bigcup_{i=2}^{r-1}V(C_s^i)\right]
       \cong C_s\BoxProd P_{r-2}.
\]
Row \(1\) lies between them, and \(e\) joins row \(0\) to row \(1\). Since \(r/2\geq2\), \eqref{eq:torus-normalization} gives \(f\in E(L^+)\).

If \(f\in E(C_s^2)\), write \(f=x_d^2x_{d\modadd{s}1}^2\) with \(d\in\{0,\ldots,s-1\}\). Choose
\[
 c=\begin{cases}
 s-1,&f\in E(C_s^2)\text{ and }d\text{ is even},\\
 0,&f\in E(C_s^2)\text{ and }d\text{ is odd},\\

 \end{cases}
\]
If \(f\notin E(C_s^2)\), then we choose \(c=0\). Take \(H^-=C_s^0\), which contains \(\mathcal M_0(c)\). The edge \(e\) is not in \(H^-\) and will be added when this cycle is extended into row \(1\). Apply Lemma~\ref{lem:cylinder} to \(L^+\) to obtain a Hamiltonian cycle \(H^+\) containing \(\mathcal M_2(c)\cup\{f\}\).

For \(j=0,\ldots,s/2-1\), define
\[
 R_j^-=R_{0,1}\bigl(c\modadd{s}(2j)\bigr),\qquad
 R_j^+=R_{2,1}\bigl(c\modadd{s}(2j)\bigr),
\]
and set
\[
 J_1=H^-\symdiff\bigcup_{j=0}^{q-1}R_j^-,\qquad
 J_2=H^+\symdiff\bigcup_{j=q}^{s/2-1}R_j^+.
\]

It is similar to Subcase 2.1: the same concatenation-(II) argument using $H^-=C_s^0$ and the column-pair partition shows that $J_1$ and $J_2$ are cycles covering all vertices of row~$1$; the length calculation is identical.
\[
V(J_1)\mathbin{\dot\cup}V(J_2)=V(T),\qquad
|J_1|=s+2q=\ell,\qquad
|J_2|=(r-2)s+(s-2q)=sr-\ell.
\]

It remains to verify the prescribed edges. Our choice gives \(c\in\{0,s-1\}\), so the first column pair \(\{c,c\modadd{s}1\}\) contains column \(0\). Hence the operation with \(R_0^-\) adds the vertical edge \(x_0^0x_0^1=e\). Since \(q\geq1\), this operation is included in the construction of \(J_1\). Every later operation on \(J_1\) uses a column pair disjoint from the first pair, so it does not affect \(e\). Therefore \(e\in E(J_1)\). The edges deleted from \(H^+\) are horizontal edges of \(\mathcal M_2(c)\). If \(f\notin E(C_s^2)\), none is \(f\). If \(f\in E(C_s^2)\), then \(c=s-1\) is odd when \(d\) is even, and \(c=0\) is even when \(d\) is odd. In either case, \(d\) and \(c\) have opposite parity, so \(f\notin\mathcal M_2(c)\). Thus \(f\) is not deleted and \(f\in E(J_2)\). The cycles \(J_1,J_2\) give the required 2-DCC.
\endgroup

\end{proof}

The next lemma supplies Hamiltonian paths through prescribed edges; it is used in Lemma~\ref{lem:product-path} to build Hamiltonian cycles in factor products.

\begin{lemma}\label{lem:path}
Let $H$ be a graph of even order $N$. Suppose that $H\cong K_2$ or $H$ is edge-Hamiltonian. Then:
\begin{enumerate}
\item[(i)] For every $F\subseteq E(H)$ with $|F|\leq2$, there is a Hamiltonian path $P$ in $H$ such that $F\subseteq E(P)$.
\item[(ii)] For every $z\in V(H)$ and $e\in E(H)$, there is a Hamiltonian path $P=zv_1\cdots v_{N-1}$ in $H$ such that $e\in E(P)$.
\end{enumerate}
\end{lemma}
\begin{proof}
If $H\cong K_2$, write $V(H)=\{x,y\}$ and $E(H)=\{xy\}$. The path $xy$ proves (i). In (ii), choose $P=xy$ when $z=x$ and $P=yx$ when $z=y$.

Assume that $H$ is edge-Hamiltonian. Then $N\geq4$.

For (i), let $F\subseteq E(H)$ with $|F|\leq2$. If $|F|\leq1$, choose a Hamiltonian cycle $C$ with $F\subseteq E(C)$ and an edge $h\in E(C)\setminus F$. The Hamiltonian path $P=C-h$ contains $F$.

Suppose that $F=\{e,f\}$ with $e\ne f$. Choose a Hamiltonian cycle $C$ containing $e$. If $f\in E(C)$, choose $h\in E(C)\setminus\{e,f\}$ and take $P=C-h$. If $f\notin E(C)$, label $C=v_0v_1\cdots v_{N-1}v_0$ so that $f=v_0v_k$, where $2\leq k\leq N-2$. Define
\begin{equation*}
\begin{aligned}
P_1&=v_1\cdots v_kv_0v_{N-1}\cdots v_{k+1},\\
P_2&=v_{k-1}\cdots v_0v_k\cdots v_{N-1},\\
D_1&=\{v_0v_1,v_kv_{k+1}\},\\
D_2&=\{v_{k-1}v_k,v_{N-1}v_0\}\,.
\end{aligned}
\end{equation*}
Both $P_1$ and $P_2$ are Hamiltonian paths in $H$, and
\begin{equation*}
E(P_1)=(E(C)\setminus D_1)\cup\{f\},\qquad
E(P_2)=(E(C)\setminus D_2)\cup\{f\}\,.
\end{equation*}
Since $D_1\cap D_2=\varnothing$, choose $P=P_1$ if $e\notin D_1$ and $P=P_2$ if $e\in D_1$. Then $\{e,f\}\subseteq E(P)$.

For (ii), fix $z\in V(H)$ and $e\in E(H)$. Choose a Hamiltonian cycle $C=zv_1\cdots v_{N-1}z$ containing $e$, and define
\begin{equation*}
P=\begin{cases}
zv_1v_2\cdots v_{N-1},&e\ne zv_{N-1},\\
zv_{N-1}v_{N-2}\cdots v_1,&e=zv_{N-1}.
\end{cases}
\end{equation*}
The path $P$ is Hamiltonian, starts at $z$, and contains $e$.
\end{proof}

Lemma~\ref{lem:product-path} is used in the proof of the main theorem to find Hamiltonian cycles in the two graphs $G_I$ and $G_J$, which together supply the spanning torus $D_I\square D_J$.

\begin{lemma}\label{lem:product-path}
Let $d\geq2$ and $G=H_1\square\cdots\square H_d$. Suppose that, for every $i\in[d]$, the factor $H_i$ is isomorphic to $K_2$ or is an edge-Hamiltonian graph of even order. Then, for every $F\subseteq E(G)$ with $|F|\leq2$, there is a Hamiltonian cycle $C$ in $G$ such that $F\subseteq E(C)$.
\end{lemma}
\begin{proof}
First let $d=2$, and write $A=H_1$, $B=H_2$, $p=|V(A)|$, and $q=|V(B)|$. Fix $F\subseteq E(A\square B)$ with $|F|\leq2$. Choose $e,f\in E(A\square B)$, not necessarily distinct, such that $F\subseteq\{e,f\}$. For $a \in V(A)$ and $b \in V(B)$, the map $(a,b)\mapsto(b,a)$ is an isomorphism from $A\square B$ to $B\square A$. By exchanging the factor labels and the edge labels when necessary, it suffices to consider the following two cases.

\smallskip
\noindent\textbf{Case 1.} $e$ and $f$ have $A$-direction.

Let $e_A,f_A\in E(A)$ be the projections of $e,f$ onto $A$. Lemma~\ref{lem:path}(i) gives a Hamiltonian path $P=a_0a_1\cdots a_{p-1}$ in $A$ with $\{e_A,f_A\}\subseteq E(P)$. If $B\cong K_2$, label $V(B)=\{b_0,b_1\}$. If $B$ is edge-Hamiltonian, choose a Hamiltonian cycle $b_0b_1\cdots b_{q-1}b_0$ in $B$. It is similar to prove for $e$ and $f$ have $B$-direction.

\smallskip
\noindent\textbf{Case 2.} $e$ has $A$-direction and $f$ has $B$-direction.

Let $e_A\in E(A)$ be the projection of $e$ onto $A$, and write $f=(z,u)(z,v)$, where $z\in V(A)$ and $uv\in E(B)$. Lemma~\ref{lem:path}(ii) gives a Hamiltonian path $P=a_0a_1\cdots a_{p-1}$ in $A$ with $a_0=z$ and $e_A\in E(P)$. If $B\cong K_2$, set $b_0=u$ and $b_1=v$. If $B$ is edge-Hamiltonian, choose a Hamiltonian cycle $b_0b_1\cdots b_{q-1}b_0$ in $B$ containing $uv$, with $b_0=u$ and $b_{q-1}=v$.

\smallskip
Figure~\ref{fig:product-snake} shows the construction, with $f$ as the closing edge in Case~2.
\begin{figure}[H]
\centering
\resizebox{0.98\linewidth}{!}{%
\begingroup%
\fontencoding{OT1}\fontfamily{cmr}\fontseries{m}\fontshape{n}\selectfont%
\renewcommand{\normalsize}{\fontsize{10}{12}\selectfont}%
\renewcommand{\small}{\fontsize{9}{11}\selectfont}%
\renewcommand{\footnotesize}{\fontsize{8}{9.5}\selectfont}%
\renewcommand{\large}{\fontsize{12}{14}\selectfont}%
\normalsize%
\definecolor{routeblue}{RGB}{30,77,155}%
\definecolor{closered}{RGB}{184,40,52}%
\definecolor{edgeorange}{RGB}{183,101,0}%
\setlength{\fboxsep}{7pt}%
\setlength{\fboxrule}{0pt}%
\raisebox{\depth}{\fbox{%
\begin{tikzpicture}[
  x=1cm,y=1cm,
  font=\small,
  vertex/.style={circle,fill=black,inner sep=1.45pt,outer sep=0pt},
  route/.style={draw=routeblue,line width=1.1pt},
  forward/.style={route,postaction={decorate},
    decoration={markings,mark=at position .58 with {\arrow{Stealth[length=4pt,width=3pt]}}}},
  reverse/.style={route,postaction={decorate},
    decoration={markings,mark=at position .42 with {\arrowreversed{Stealth[length=4pt,width=3pt]}}}},
  omitted/.style={route,dash pattern=on .45pt off 3pt,line cap=round},
  coordlabel/.style={font=\small,inner sep=1pt,fill=white},
  panelhead/.style={font=\small\bfseries},
]

\foreach \row/\yy/\bb/\parity in {
  0/0/0/0,
  1/-1.18/1/1,
  2/-2.36/2/0,
  3/-4.00/{q-3}/1,
  4/-5.18/{q-2}/0,
  5/-6.36/{q-1}/1
}{
  \coordinate (v\row-0) at (0,\yy);
  \coordinate (v\row-1) at (2.6,\yy);
  \coordinate (v\row-2) at (6.4,\yy);
  \coordinate (v\row-3) at (9.6,\yy);
  \ifnum\parity=0
    \draw[forward] (v\row-0)--(v\row-1);
    \draw[omitted] (v\row-1)--(v\row-2);
    \draw[forward] (v\row-2)--(v\row-3);
  \else
    \draw[reverse] (v\row-0)--(v\row-1);
    \draw[omitted] (v\row-1)--(v\row-2);
    \draw[reverse] (v\row-2)--(v\row-3);
  \fi
  \foreach \col in {0,1,2,3}{
    \node[vertex] at (v\row-\col) {};
  }
  \node[coordlabel,anchor=south west,xshift=2pt,yshift=5pt]
    at (v\row-0) {$(a_0,b_{\bb})$};
  \node[coordlabel,anchor=south,yshift=5pt]
    at (v\row-1) {$(a_1,b_{\bb})$};
  \node[coordlabel,anchor=south,yshift=5pt]
    at (v\row-2) {$(a_{p-2},b_{\bb})$};
  \node[coordlabel,anchor=south east,xshift=-2pt,yshift=5pt]
    at (v\row-3) {$(a_{p-1},b_{\bb})$};
}

\draw[forward] (v0-3)--(v1-3);
\draw[forward] (v1-0)--(v2-0);
\draw[forward] (v3-0)--(v4-0);
\draw[forward] (v4-3)--(v5-3);
\draw[omitted] (v2-3)--(v3-3);
\node[font=\large,text=black!55] at (.65,-3.18) {$\vdots$};

\draw[draw=closered,line width=1.2pt,rounded corners=5pt,
  -{Stealth[length=5pt,width=4pt]}]
  (v5-0)--(-1.23,-6.36)--(-1.23,0)--(v0-0);
\node[text=closered,anchor=east,inner sep=2pt]
  at (-1.32,-3.18) {$f$};

\coordinate (eleft) at (2.95,-3.18);
\coordinate (eright) at (6.65,-3.18);
\draw[omitted] (1.35,-3.18)--(eleft);
\draw[omitted] (eright)--(8.25,-3.18);
\draw[draw=edgeorange,line width=1.6pt] (eleft)--(eright);
\node[vertex] at (eleft) {};
\node[vertex] at (eright) {};
\node[coordlabel,anchor=south west,yshift=4pt] at (eleft) {$(a_k,b_j)$};
\node[coordlabel,anchor=south east,yshift=4pt] at (eright) {$(a_{k+1},b_j)$};
\node[text=edgeorange,anchor=north,inner sep=1pt,yshift=-3pt] at (4.8,-3.18) {$e$};

\end{tikzpicture}%
}}%
\endgroup%
}%
\caption{The serpentine construction gives a Hamiltonian cycle containing $e$ and $f$.}
\label{fig:product-snake}
\end{figure}

For both cases, $q$ is even, $b_jb_{j+1}\in E(B)$ for $0\leq j<q-1$, and $b_{q-1}b_0\in E(B)$. For $0\leq j<q$, define
\begin{equation}\label{eq:product-cycle}
\begin{aligned}
P_j&=\begin{cases}
(a_0,b_j)(a_1,b_j)\cdots(a_{p-1},b_j),&j\text{ is even},\\
(a_{p-1},b_j)(a_{p-2},b_j)\cdots(a_0,b_j),&j\text{ is odd},
\end{cases}\\
C&=P_0P_1\cdots P_{q-1}(a_0,b_0)\,.
\end{aligned}
\end{equation}
The paths $P_0,\ldots,P_{q-1}$ are pairwise vertex-disjoint, and $\bigcup_{j=0}^{q-1}V(P_j)=V(A)\times V(B)$. For $0\leq j<q$,
\begin{equation}\label{eq:product-row-edges}
E(P_j)=\{(a_i,b_j)(a_{i+1},b_j):0\leq i<p-1\}
\subseteq E(A\square B)\,.
\end{equation}
For $0\leq j<q-1$, join the last vertex of $P_j$ to the first vertex of $P_{j+1}$ using
\begin{equation*}
\begin{cases}
(a_{p-1},b_j)(a_{p-1},b_{j+1}),&j\text{ is even},\\
(a_0,b_j)(a_0,b_{j+1}),&j\text{ is odd}.
\end{cases}
\end{equation*}
Since $q$ is even, $(a_0,b_{q-1})(a_0,b_0)\in E(A\square B)$. Hence $C$ is a Hamiltonian cycle in $A\square B$. In Case 1, equation~\eqref{eq:product-row-edges} gives $e,f\in E(C)$. In Case 2, equation~\eqref{eq:product-row-edges} gives $e\in E(C)$, and $f=(a_0,b_{q-1})(a_0,b_0)\in E(C)$. Thus $F\subseteq E(C)$.

For $d\geq3$, assume that the result holds for products of $d-1$ factors, and put $A=H_1\square\cdots\square H_{d-1}$. The induction hypothesis implies that $A$ is edge-Hamiltonian. The order $|V(A)|=\prod_{i=1}^{d-1}|V(H_i)|$ is even. The two-factor case applied to $A$ and $H_d$ proves the result for $G=A\square H_d$.
\end{proof}

\subsection{Proof of Theorem~\ref{thm:product}}
Put $N=|V(G)|$. Every factor has even order, so $N$ is even. Fix independent edges $e,f\in E(G)$ and an even integer $\ell$ with $4\leq\ell\leq N/2$. We construct vertex-disjoint cycles $J_1,J_2$ in $G$ satisfying
\begin{equation}\label{eq:main-cover}
\begin{gathered}
V(J_1)\mathbin{\dot\cup}V(J_2)=V(G),\qquad
|J_1|=\ell,\qquad |J_2|=N-\ell,\\
e\in E(J_1),\qquad f\in E(J_2)\,.
\end{gathered}
\end{equation}

For (i)--(iv), Lemma~\ref{lem:torus} reduces the proof to finding a spanning subgraph $T\subseteq G$ such that $T\cong C_s\square C_r$ and $\{e,f\}\subseteq E(T)$, where $s,r\geq4$ are even.

We use the notation $G_K$, $\pi_K$, and $F_K$ of Section~\ref{subsec:projections}; for disjoint nonempty $I,J\subseteq[d]$ with $I\cup J=[d]$, the map $u\mapsto(\pi_I(u),\pi_J(u))$ is a graph isomorphism from $G$ to $G_I\square G_J$.
Suppose $D_I$ is a Hamiltonian cycle in $G_I$ with $F_I\subseteq E(D_I)$ and $D_J$ is a Hamiltonian cycle in $G_J$ with $F_J\subseteq E(D_J)$. Put $T=D_I\square D_J$, $s=|V(G_I)|$, $r=|V(G_J)|$. Then $T\subseteq G$, $T\cong C_s\square C_r$, $V(T)=V(G)$, and $sr=N$. For each $g=uv\in\{e,f\}$, let $K\in\{I,J\}$ contain the direction of $g$, and put $L=[d]\setminus K$. Then $\pi_K(u)\pi_K(v)\in F_K\subseteq E(D_K)$ and $\pi_L(u)=\pi_L(v)\in V(D_L)$, so $g\in E(T)$. Thus $\{e,f\}\subseteq E(T)$. It remains to choose $I,J,D_I,D_J$ with $s,r\geq4$.
\begin{enumerate}
\item[(i)] Set $I=\{1,2\}$ and $J=\{3,\ldots,d\}$. Since $|I|,|J|\geq2$, for each $K\in\{I,J\}$, Lemma~\ref{lem:product-path} applied to $G_K$ and $F_K$ gives a Hamiltonian cycle $D_K$ in $G_K$ with $F_K\subseteq E(D_K)$.

\item[(ii)] Choose distinct $a,b\in[3]$ with $|V(H_a)|,|V(H_b)|\geq4$. Each of $e,f$ has only one direction, so
\begin{equation*}
|F_{\{a\}}|+|F_{\{b\}}|\leq2\,.
\end{equation*}
Choose $k\in\{a,b\}$ with $|F_{\{k\}}|\leq1$, and set $I=\{k\}$ and $J=[3]\setminus\{k\}$. Since $|F_I|\leq1$ and $H_k$ is edge-Hamiltonian, choose a Hamiltonian cycle $D_I$ in $G_I=H_k$ with $F_I\subseteq E(D_I)$. Apply Lemma~\ref{lem:product-path} to $G_J$ and $F_J$ to choose a Hamiltonian cycle $D_J$ in $G_J$ with $F_J\subseteq E(D_J)$.

\item[(iii)] Let $k$ index the unique factor of order at least four, and set $I=\{k\}$ and $J=[3]\setminus\{k\}$. By assumption (iii), choose a Hamiltonian cycle $D_I$ in $G_I=H_k$ with $F_I\subseteq E(D_I)$. Since $G_J\cong K_2\square K_2\cong C_4$, take $D_J=G_J$.

\item[(iv)] Set $I=\{1\}$ and $J=\{2\}$. For each $K\in\{I,J\}$, assumption (iv) gives a Hamiltonian cycle $D_K$ in $G_K$ with $F_K\subseteq E(D_K)$.
\end{enumerate}

In each case, $s,r$ are even and at least four, so $N=sr\geq16$. The edges $e,f$ are independent in $T$, and $4\leq\ell\leq N/2=sr/2$. Lemma~\ref{lem:torus} gives vertex-disjoint cycles $J_1,J_2$ in $T\subseteq G$ satisfying~\eqref{eq:main-cover}.

In (v), fix a graph isomorphism from $G$ to $Q_3$ on vertex set $\{0,1\}^3$. Write $u_1u_2u_3$ for the vertex $(u_1,u_2,u_3)$. Here $N=8$ and $\ell=4$. By permuting coordinates and replacing coordinates $u_i$ by $1-u_i$, we may assume that $e=000\,100$. Here $1\leq i\leq3$. The following table lists all seven possible edges $f$ independent of $e$, together with vertex-disjoint $4$-cycles $J_1,J_2$ satisfying~\eqref{eq:main-cover}.
\begin{center}
\setlength{\tabcolsep}{9pt}
\renewcommand{\arraystretch}{1.15}
\begin{tabular}{ccc}
$f$ & $J_1$ & $J_2$\\
\hline
$010\,110$ & $000\,100\,101\,001\,000$ & $010\,110\,111\,011\,010$\\
$001\,101$ & $000\,100\,110\,010\,000$ & $001\,101\,111\,011\,001$\\
$011\,111$ & $000\,100\,110\,010\,000$ & $001\,101\,111\,011\,001$\\
$001\,011$ & $000\,100\,110\,010\,000$ & $001\,101\,111\,011\,001$\\
$101\,111$ & $000\,100\,110\,010\,000$ & $001\,101\,111\,011\,001$\\
$010\,011$ & $000\,100\,101\,001\,000$ & $010\,110\,111\,011\,010$\\
$110\,111$ & $000\,100\,101\,001\,000$ & $010\,110\,111\,011\,010$\\
\end{tabular}
\end{center}

Finally, if every factor is bipartite, then $G$ is bipartite. By Definition~\ref{def:2dcc}, the cycles in~\eqref{eq:main-cover} establish 2-DCC edge $[4,N/2]$-bipancyclicity.

\begingroup
\color{black}
\hypersetup{linkcolor=black,citecolor=black}
\begin{remark}\label{rem:edge-vertex}
Under the hypotheses of Theorem~\ref{thm:product}, $G$ is also 2-DCC vertex $[4,|V(G)|/2]$-bipancyclic and 2-DCC $[4,|V(G)|/2]$-bipancyclic.
\end{remark}
\endgroup

\section{Applications}\label{sec:applications}
We apply Theorem~\ref{thm:product} to the three graph families defined in Section~\ref{subsec:families} (Figure~\ref{fig:examples}) and determine the exceptional cases in each.

\subsection{Undirected toroidal meshes}\label{subsec:generalized}
Let $C(d_1,\ldots,d_n)$ be the undirected toroidal mesh of Definition~\ref{def:mesh}.

\begin{corollary}\label{cor:generalized}
Let $G=C(d_1,\ldots,d_n)$ with $n\geq2$ and all $d_i\geq2$ even. Here $1\leq i\leq n$. Then $G$ is 2-DCC edge $[4,|V(G)|/2]$-bipancyclic if and only if
\[
G\not\cong C_4
\quad\text{and}\quad
G\not\cong K_2\square C_{2p}\quad\text{for every }p\geq3.
\]
\end{corollary}
\begin{proof}
If $n\geq4$, apply Theorem~\ref{thm:product}(i). If $n=3$ and at least two factors have order at least four, use case~(ii). If $n=3$ and exactly one factor has order at least four, every pair of edges of that factor lies on a common Hamiltonian cycle, so case~(iii) applies. If $n=3$ and every factor is $K_2$, use case~(v). If $n=2$ and both factors have order at least four, use case~(iv). If $n=2$, one factor is $K_2$ and the other factor is be $C_4$, then the product is $K_2\square C_4\cong Q_3$, which falls under case~(v). Otherwise, the graph $K_2\square K_2\cong C_4$ has fewer than eight vertices. For $K_2\square C_{2p}$ with $p\geq3$, write $G=C_{2p}\square K_2$ and let $uv\in E(C_{2p})$. Let $e=(u,0)(v,0)$ and $f=(u,1)(v,1)$. These are independent edges of $G$. Because $C_{2p}$ has no $4$-cycle, the unique $4$-cycle of $G$ containing $e$ is $(u,0)(v,0)(v,1)(u,1)(u,0)$, which contains both endpoints of $f$. Thus $e$ cannot lie in a $4$-cycle vertex-disjoint from a cycle containing $f$.
\end{proof}

\begingroup
\color{black}
\hypersetup{linkcolor=black,citecolor=black}
\begin{remark}\label{rem:mesh-vertex}
By Definition~\ref{def:2dcc}, Corollary~\ref{cor:generalized} also implies that the corresponding networks are 2-DCC vertex-bipancyclic and 2-DCC bipancyclic, and hence implies the main result of Niu et al.~\cite{NiuXuLai2021} on 2-DCC vertex-bipancyclicity.
\end{remark}
\endgroup

In particular, $Q_n$ is 2-DCC edge bipancyclic if and only if $n\geq3$. Every bipartite $k$-ary $n$-cube with $k\geq4$ and $n\geq2$ is 2-DCC edge bipancyclic.

\subsection{Star-cube networks}\label{subsec:starcube}
Let $\mathit{SC}_{m,n}=S_m\square Q_n$ be the star-cube network of Definition~\ref{def:starcube}.

\begin{corollary}\label{cor:starcube}
For $m\geq3$ and $n\geq1$, the star-cube network $\mathit{SC}_{m,n}$ is 2-DCC edge $[4,m!2^{n-1}]$-bipancyclic if and only if $n\geq2$.
\end{corollary}
\begin{proof}
\emergencystretch=1em
If $n=2$, use Theorem~\ref{thm:product}(iii) directly on $S_m\square K_2\square K_2$; for $m\geq4$ the two-edge Hamiltonian-cycle property of $S_m$ is given by~Lemma~\ref{lem:star-two-edges}, and for $m=3$ it follows from $S_3\cong C_6$, which has a Hamiltonian cycle through every pair of its edges. If $n\geq3$, use case~(i). If $n=1$, write $G=S_m\square K_2$ and let $uv\in E(S_m)$. Let $e=(u,0)(v,0)$ and $f=(u,1)(v,1)$. These are independent edges of $G$. Since $S_m$ has no $4$-cycle, the unique $4$-cycle of $G$ containing $e$ is $(u,0)(v,0)(v,1)(u,1)(u,0)$, which contains both endpoints of $f$. Thus $e$ cannot lie in a $4$-cycle vertex-disjoint from a cycle containing $f$.
\end{proof}

\subsection{Hyperstar networks}\label{subsec:hyperstar}
Let $S_{m_1,\ldots,m_r}$ and $N$ be as in Definition~\ref{def:hyperstar}.

\begin{corollary}\label{cor:hyperstar}
The graph $S_{m_1,\ldots,m_r}$ is 2-DCC edge $[4,N/2]$-bipancyclic if and only if $r\geq2$.
\end{corollary}
\begin{proof}
Here $1\leq i\leq r$. Each $S_{m_i}$ has a Hamiltonian cycle through every pair of its edges (by $S_3\cong C_6$ and~Lemma~\ref{lem:star-two-edges}), so Theorem~\ref{thm:product}(iv) applies when $r=2$, case~(ii) applies when $r=3$, and case~(i) applies when $r\geq4$. If $r=1$, the graph is $S_{m_1}$, which has no $4$-cycle. If $m_1=3$, the graph $S_3\cong C_6$ has only six vertices, so $N/2=3<4$ and the interval $[4,N/2]$ is empty.
\end{proof}


\clearpage

\begin{thebibliography}{99}
\raggedright
\bibitem{AkersKrishnamurthy1989}
S.~B.~Akers and B.~Krishnamurthy,
A group-theoretic model for symmetric interconnection networks,
\emph{IEEE Trans. Comput.} \textbf{38} (1989), 555--566.

\bibitem{AlAyyoubDay1998}
A.-E.~Al-Ayyoub and K.~Day,
The hyperstar interconnection network,
\emph{J. Parallel Distrib. Comput.} \textbf{48} (1998), 175--199.

\bibitem{DayAlAyyoub1997}
K.~Day and A.-E.~Al-Ayyoub,
The cross product of interconnection networks,
\emph{IEEE Trans. Parallel Distrib. Syst.} \textbf{8} (1997), 109--118.

\bibitem{HammackImrichKlavzar2011}
R.~Hammack, W.~Imrich, and S.~Klav\v{z}ar,
\emph{Handbook of Product Graphs}, 2nd ed., CRC Press, Boca Raton, FL, 2011.

\bibitem{HaoQinZhangChang2024}
R.-X.~Hao, X.-W.~Qin, H.~Zhang, and J.-M.~Chang,
Two-disjoint-cycle-cover pancyclicity of data center networks,
\emph{Appl. Math. Comput.} \textbf{475} (2024), 128716.

\bibitem{KungChen2015}
T.-L.~Kung and H.-C.~Chen,
Complete cycle embedding in crossed cubes with two-disjoint-cycle-cover pancyclicity,
\emph{IEICE Trans. Fundam. Electron. Commun. Comput. Sci.} \textbf{E98-A} (2015), 2670--2676.

\bibitem{KungChenLinHsu2021}
T.-L.~Kung, H.-C.~Chen, C.-H.~Lin, and L.-H.~Hsu,
Three types of two-disjoint-cycle-cover pancyclicity and their applications to cycle embedding in locally twisted cubes,
\emph{Comput. J.} \textbf{64} (2021), 27--37.

\bibitem{LiChenLu2025}
H.~Li, L.~Chen, and M.~Lu,
Two-disjoint-cycle-cover pancyclicity of split-star networks,
\emph{Appl. Math. Comput.} \textbf{487} (2025), 129085.

\bibitem{NiuXuLai2021}
R.~Niu, M.~Xu, and H.-J.~Lai,
Two-disjoint-cycle-cover vertex bipancyclicity of the bipartite generalized hypercube,
\emph{Appl. Math. Comput.} \textbf{400} (2021), 126090.

\bibitem{NiuZhouXu2023}
R.~Niu, S.~Zhou, and M.~Xu,
Two-disjoint-cycle-cover vertex bipancyclicity of bipartite hypercube-like networks,
\emph{Theor. Comput. Sci.} \textbf{947} (2023), 113703.

\bibitem{WeiHaoChang2020}
C.~Wei, R.-X.~Hao, and J.-M.~Chang,
Two-disjoint-cycle-cover bipancyclicity of balanced hypercubes,
\emph{Appl. Math. Comput.} \textbf{381} (2020), 125305.

\bibitem{WuSabir2023}
W.~Wu and E.~Sabir,
Embedding spanning disjoint cycles in hypercube networks with prescribed edges in each cycle,
\emph{Axioms} \textbf{12} (2023), 861.

\bibitem{XueLuQiao2025}
S.~Xue, Z.~P.~Lu, and H.~Qiao,
Two-disjoint-cycle-cover edge/vertex bipancyclicity of star graphs,
\emph{Discrete Appl. Math.} \textbf{360} (2025), 196--208.

\bibitem{ZhangHaoLaiLee2023}
H.~Zhang, R.-X.~Hao, H.-J.~Lai, and J.~Lee,
Two-disjoint-cycle-cover bipancyclicity of bubble-sort star graphs,
\emph{Discrete Appl. Math.} \textbf{338} (2023), 320--328.

\end{thebibliography}
\end{document}